\documentclass[12pt]{amsart}

\usepackage{amsmath}
\usepackage{amsfonts}
\usepackage{amssymb,enumerate}
\usepackage{amsthm}
\usepackage{fancyhdr}
\usepackage{hyperref}

\newtheorem{lemma}{Lemma}
\newtheorem{proposition}[lemma]{Proposition}
\newtheorem{theorem}[lemma]{Theorem}
\newtheorem{question}[lemma]{Question}

\newcommand{\ap}{\mathrm {Ap}}
\newcommand{\Sw} {\mathcal{S}}
\newcommand{\edim}{\mathrm {edim}}

\title{On a conjecture by Wilf about the Frobenius number}

\author{Alessio Moscariello \and Alessio Sammartano}



\address[Alessio Moscariello]{Dipartimento di Matematica e Informatica, Universit\`a di Catania, Viale Andrea Doria 6, 95125 Catania, Italy}

\email{alessio.moscariello@studium.unict.it}

\address[Alessio Sammartano]{Department of Mathematics, Purdue University, 150 N. University Street, West Lafayette, IN 47907, USA}

\email{asammart@math.purdue.edu}

\begin{document}

\begin{abstract}
Given   coprime positive integers $a_1 <  \cdots < a_d$, the Frobenius number $F$ is the largest integer which is not representable as a non-negative integer combination of the $a_i$.
Let $g$ denote the number of all non-representable positive integers:
 Wilf conjectured that $d \geq \frac{F +1}{F+1-g}$.
We prove that for every fixed value of $ \lceil \frac{a_1}{d} \rceil $ the conjecture holds for all  values of $a_1$ which are sufficiently large  and are not divisible by a finite set of primes.
We also propose a generalization in the context of one-dimensional local rings  and a question on the equality $d = \frac{F+1}{F+1-g}$.
\end{abstract}

\maketitle

\section*{Introduction}

The Diophantine Frobenius problem or coin problem consists of studying the greatest integer $F$,
 called the Frobenius number,
 that is not representable as a linear combination of given $d$ coprime positive integers $a_1<  \cdots < a_d$ with coefficients in $\mathbb{N}$.
The problem has stimulated much research over the past decades,  due to the applications to several areas of pure and applied mathematics including
 Coding Theory,  Linear Algebra, Combinatorics and Commutative Algebra.
The monograph \cite{RamirezAlfonsin} explores several viewpoints about the problem.

Let $ \Sw = \big\{ \sum_{i=1}^d \lambda_i a_i \, \colon \, \lambda_i \in \mathbb{N} \big\} $
 be the  numerical semigroup generated by the $a_i$.
The condition $\gcd(a_1, \ldots,  a_d)=1$ implies that the set of gaps $\mathbb{N}\setminus \Sw$ is  finite and its cardinality  
$g = | \mathbb{N}\setminus \Sw|$ is known as the genus of $\Sw$. 
In 1978 H. S. Wilf proposed a lower bound for the number of generators of $\Sw$ in terms of the Frobenius number and the genus (cf. \cite{Wilf}): 
\begin{equation}\label{Wilf's Inequality}
d \geq \frac{F+1}{F+1-g}.
\end{equation}
Although the problem has been considered by several authors 
(cf. 
\cite{Barucci}, \cite{BrasAmoros}, \cite{DobbsMatthews}, \cite{FGH}, \cite{Kaplan}, \cite{Sammartano}, \cite{Zhai})
only special cases have been solved and it remains wide open.
The approach we follow in this paper is to study the poset structure of the so-called Ap\'ery set
once the value of  the ratio $\rho =  \lceil \frac{a_1}{d} \rceil $ is fixed. The cases 
$\rho=1$ and $\rho = 2$ were solved in \cite{DobbsMatthews} and \cite{Sammartano} respectively; in this work we focus on the general case.
Our main result is the following:

\begin{theorem}\label{Main Theorem}
For every value of $\rho =  \lceil \frac{a_1}{d} \rceil $ Wilf's inequality (\ref{Wilf's Inequality})
holds if $a_1$ is large enough and the prime factors of $a_1$ are greater than or equal to $\rho$.
\end{theorem}
\noindent
The explicit lower bound for $a_1$ in Theorem \ref{Main Theorem} is $a_1 \geq \frac{\rho (3 \rho^2 - \rho -4) (3 \rho^2 - \rho -2)}{8(\rho-2)}$.

In the first section of the paper we recall some definitions and preliminary results, 
whereas the second section is devoted to proving Theorem \ref{Main Theorem}.
We conclude the paper with a word about  the equality in (\ref{Wilf's Inequality}) and a discussion of the conjecture in the context of Commutative Algebra.

\section{Preliminaries} 

We assume without loss of generality that 
$a_1, \ldots, a_d$ are the minimal generators of $\Sw$, i.e. no proper subset generates the semigroup; 
this implies that $d \leq a_1$.
The Ap\'ery set  is defined as 
$$
\ap(\Sw)=\big\{\omega \in \Sw \ \colon \ \omega - a_1 \not \in \Sw\big\}
$$
and thus it consists of the smallest elements of $\Sw$ in each residue class modulo $a_1$.
It follows that $0 \in \ap(\Sw)$, $|\ap(\Sw)|=a_1$
and that $F+a_1$ is the largest element in $\ap(\Sw)$.
We list the elements  of $\ap(\Sw)$  increasingly setting 
$\ap(\Sw)=\{\omega_0 < \omega_1 < \cdots < \omega_{a_1-1} \}$,
notice that we have $\omega_0=0,\, \omega_1=a_2, \, \omega_{a_1-1}=F+a_1$.
We establish a partial order on  $\mathbb{N}$ by setting $n_1 \preceq n_2$ if there exists $s \in \Sw$ such that $n_1+s=n_2$ and we consider throughout the paper the poset $(\ap(\Sw)\setminus\{0\}, \preceq)$. 
The minimal elements in $(\ap(\Sw)\setminus\{0\}, \preceq)$ are exactly $\{a_2, \ldots, a_d\}$ and if $\tau \in \Sw, \omega \in \ap(\Sw), \, \tau \preceq \omega$ then $\tau \in \ap(\Sw)$.
We refer to  \cite{RosalesGarciaSanchez} for  details on numerical semigroups.

Now we give further definitions in order to introduce a reformulation of Wilf's inequality.
For each $k\in \mathbb{N}$ let $I_k=\big[ka_1, (k+1)a_1-1\big] $  and $n_k=\big|\Sw \cap \big[0, F\big] \cap I_k \big|$. 
We write 
$$
F+1 = Qa_1+R
$$ 
with   $Q, R \in \mathbb{N}$ and  $2\leq R \leq a_1$.
Note that $R \ne 1$ as $F \notin \Sw$ and that $I_Q$ is the  interval containing the Frobenius number.
We define the numbers 
$$
\eta_j  =  \Big|\big\{k \in \mathbb{N}\ \colon \big|I_k \cap \Sw\big|=j\big\}\Big|
   	\qquad \text{and}  \qquad
\epsilon_j  = \Big|\big\{k \in \mathbb{N}\ \colon \big|I_k \cap \Sw\big|=j,\, 0 \leq k \leq Q-1 \big\}\Big|
$$
for each $j \in \big\{1,2,\ldots, a_1-1\big\}$.
In other words, 
$\eta_j$ counts the intervals $I_k$ containing exactly $j$ elements of $\Sw$,
while 
$\epsilon_j$ 
only counts such intervals among the first $Q$.
The two definitions differ slightly, and the numbers $\eta_j$ can be expressed in terms of $\ap(\Sw)$:
\begin{lemma}[\cite{Sammartano},  12, 13]\label{Lemma Properties of Eta and Epsilon}
The following properties hold:
\begin{enumerate}
\item $\epsilon_j = \eta_j-1$ if $j=\big|I_Q \cap \Sw\big|$ and  $\epsilon_j = \eta_j$ otherwise;
\item $\eta_j= \lfloor \frac{\omega_j}{a_1}  \rfloor -  \lfloor \frac{\omega_{j-1}}{a_1}  \rfloor$ for all $1 \le j \le a_1-1$.
\end{enumerate}
\end{lemma}
The numbers $\epsilon_j$ give rise to an equivalent formulation of (\ref{Wilf's Inequality}):
\begin{proposition}[\cite{Sammartano}, 10, 11]\label{Proposition Expression Difference} 
We have the equation
$$
d(F+1-g) - (F+1)=
\sum_{j=1}^{a_1-1} \epsilon_j(jd-a_1)+(n_Q d-R) = : \Delta.
$$
\end{proposition}

\section{Proof of the main theorem}

In this section we  prove Theorem \ref{Main Theorem} using Proposition \ref{Proposition Expression Difference}. 
We need a series of lemmas first. 

\begin{lemma}\label{Lemma with X}
If  $x \in \mathbb{N}\setminus\{0\}$ is smaller than every prime factor of $a_1$ then 
$  \lfloor \frac{\omega_{x}}{a_1}  \rfloor \leq x  \lfloor \frac{\omega_{1}}{a_1}  \rfloor+x -1 $.
\end{lemma}
\begin{proof}
The assumption on $x$ implies that 
$i\omega_1 \not \equiv j\omega_1 \pmod{a_1}$ for any $0 \leq i < j \le x$:
in fact  $i\omega_1 \equiv j\omega_1 \pmod{a_1}$ yields the contradiction $(j-i)\omega_1 \equiv 0 \pmod{a_1}$ as $\gcd( j-i, a_1) = 1$ and $\omega_1 \not\equiv 0 \pmod{a_1}$.
 Thus the subset $\{\omega_1, 2\omega_1, \ldots, x\omega_1 \} \subseteq \Sw$ covers $x$ different residue classes modulo $a_1$,
  hence there are elements of at least $x$ different classes less than or equal to $x\omega_1$ in $\Sw$. 
 By definition of Ap\'ery set we deduce that $\omega_x \le x\omega_1$ and  $ \lfloor \frac{\omega_{x}}{a_1} \rfloor \le \lfloor \frac{x\omega_{1}}{a_1} \rfloor \le x \lfloor \frac{\omega_{1}}{a_1} \rfloor+x-1$.
\end{proof}

\begin{lemma}\label{Lemma With Y}
If $y\in \mathbb{N}$ satisfies $y \geq 2$ and
  $a_1-d \geq {y\choose 2}+1 $ then 
 $\omega_{a_1-1} \ge \omega_y + \omega_1$. 
In particular,  $ \lfloor \frac{\omega_{a_1-1}}{a_1}  \rfloor \ge  \lfloor \frac{\omega_{y}}{a_1}  \rfloor + \lfloor \frac{\omega_{1}}{a_1} \rfloor$ and $F>\omega_y$.
\end{lemma}
\begin{proof}
Since the minimal elements in  $(\ap(\Sw)\setminus\{0\}, \preceq)$ are exactly $\{a_2, \ldots, a_d\}$ and $|\ap(\Sw)|= a_1$, there are  at least ${y\choose 2}+1$ non-minimal elements in  $(\ap(\Sw)\setminus\{0\}, \preceq)$.
The set  $\mathcal{Y}=\{ \omega_i+\omega_j \ | \ 1 \le i \le j \le y-1 \}$
contains at most ${y\choose 2}$ distinct elements, 
so there exists $\tau \in \ap(\Sw) \setminus (\{0, a_2, \ldots, a_d\} \cup \mathcal{Y})$. 
In particular $\tau$ is not minimal, 
thus  $\tau= \omega_h + \omega_k$ for some $1 \leq h \leq k$, 
and $\tau\notin \mathcal{Y}$ yields $k \geq y$.
 Finally we have 
 $\omega_{a_1-1} \ge \tau = \omega_h + \omega_k \geq \omega_1 + \omega_y$.
The inequality with the floor function follows immediately, 
and from $\omega_1 > a_1$ we obtain $\omega_{a_1-1} = F + a_1 \geq \omega_y + \omega_1 > \omega_y +a_1$ and $F > \omega_y$. 
\end{proof}

\begin{lemma}\label{Lemma with NQ}
If $y, z\in \mathbb{N}$ satisfy $y\geq 2$, $y\geq z$,
  $ a_1 - d \geq {y\choose 2}+1 $,  and  
$ \lfloor \frac{\omega_{a_1-1}}{a_1}  \rfloor = \lfloor \frac{\omega_z}{a_1} \rfloor + \lfloor \frac{\omega_{1}}{a_1} \rfloor$  
  then $n_Q \ge y-z+3$.
\end{lemma}
\begin{proof}
Fix $z \leq i \leq y$, by Lemma
\ref{Lemma With Y} we have  that $F+a_1 \ge \omega_1+\omega_y \ge \omega_1+\omega_i$ and 
$$
 \left \lfloor \frac{\omega_{a_1-1}}{a_1} \right \rfloor 
\ge 
 \left \lfloor \frac{\omega_{y}}{a_1}  \right\rfloor  +  \left \lfloor \frac{\omega_{1}}{a_1}  \right \rfloor 
\ge 
\left \lfloor \frac{\omega_{i}}{a_1} \right \rfloor + \left \lfloor \frac{\omega_{1}}{a_1}   \right\rfloor  
 \ge 
  \left \lfloor \frac{\omega_{z}}{a_1} \right \rfloor  + \left \lfloor \frac{\omega_{1}}{a_1}   \right\rfloor 
 = 
\left  \lfloor \frac{\omega_{a_1-1}}{a_1}  \right\rfloor 
 $$
  and this leads to $ \lfloor \frac{\omega_{y}}{a_1} \rfloor= \lfloor \frac{\omega_{i}}{a_1}  \rfloor= \lfloor \frac{\omega_{z}}{a_1}  \rfloor$.
 From $F+1=Q a_1+R$ we obtain  $(Q+1)a_1+(R-1) \ge \omega_1+ \omega_i$.
Dividing by $a_1$ we get $\omega_1 = q_1 a_1 + r_1$ and $\omega_i = q_i a_1 + r_i$
with $0 \le r_i \le a_1-1$ 
and it follows  that $ Q +1 = \lfloor \frac{\omega_{a_1-1}}{a_1}  \rfloor= \lfloor \frac{\omega_{z}}{a_1}  \rfloor +  \lfloor \frac{\omega_{1}}{a_1} \rfloor  =\lfloor \frac{\omega_{i}}{a_1}  \rfloor +  \lfloor \frac{\omega_{1}}{a_1}   \rfloor = q_1 + q_i$. 
From $\omega_{a_1-1} \ge \omega_i + \omega_1$ we obtain $(Q+1)a_1+(R-1) \ge (q_i+q_1)a_1+(r_i+r_1)= (Q+1)a_1+(r_i+r_1)$, 
thus
$R-1 \ge r_i+r_1$. 
In particular, $r_1, r_i < R$:
we conclude that the elements $Qa_1, \omega_1 + l_1 a_1$ and $\omega_i + l_i a_1$ belong to $\Sw \cap I_Q \cap \big[0,F\big]$ for suitable $l_1, l_i \in \mathbb{N}$ for each $z\leq i \leq y$ so that $n_Q \geq 3+ y-z$.
\end{proof}

\begin{lemma}\label{Lemma with NQ minus 1}
If $y\in \mathbb{N}$ satisfies $y \geq 2$ and
  $  a_1 - d \geq {y\choose 2}+1$ then $n_{Q-1} \ge y+2-n_Q$.
\end{lemma}

\begin{proof}
By Lemma \ref{Lemma With Y} we have  $a_1<\omega_y < F$, 
whence $Q> 0$  and $\big\{\omega_0, \ldots, \omega_y\big\} \subseteq \big[0,F\big]$.
At most $n_Q-1$ of these elements belong to 
$\Sw \cap I_Q \cap \big[0,F\big]$, 
because $Qa_1  \notin \ap(\Sw)$.
Hence there are at least $y+2-n_Q$ elements $\omega_i$ smaller than $Qa_1$,
that is,
 $\omega_0 <\cdots < \omega_{y-n_Q+1} < Qa_1$:
 we conclude that  $\omega_i+l_i a_1 \in I_{Q-1}$ for suitable $l_i \in \mathbb{N}$ for each $0\leq i \leq y-n_Q+1$
 so  $n_{Q-1} \ge y+2-n_Q$.
\end{proof}

We are now ready to prove the main result.

\begin{proof}[Proof of Theorem \ref{Main Theorem}]
Let $\rho =  \lceil \frac{a_1}{d} \rceil$.
Since the cases $\rho = 1,  2$ have been solved,
we assume $\rho \geq 3$. 
Consider the integers
 $$
 y=\frac{3\rho^2-\rho-4}{2}
\qquad \text{ and } \qquad 
 z=\frac{\rho^2+\rho-2}{2}
 $$
 notice that $y \geq z \geq \rho +2$.
 We are going to make the assumption that $ a_1 \geq \frac{\rho}{\rho-2} {y \choose 2}$,
 so that the condition appearing in Lemmas \ref{Lemma With Y}, \ref{Lemma with NQ}, \ref{Lemma with NQ minus 1} is satisfied: 
 $$
a_1 - d
\geq  
(\rho -1)d+1  - d 
=
 (\rho-2)d+1
\geq
 \frac{\rho-2}{\rho}a_1+1
\geq
 {y \choose 2}+1.
$$
Our task is to prove (\ref{Wilf's Inequality}) by showing that the quantity
 $\Delta$ of Proposition \ref{Proposition Expression Difference} is non-negative;
we will actually show that $\Delta >0$.
We are going to break the summation in two parts, at the index $j=\rho$.
For the first part we use both properties of Lemma \ref{Lemma Properties of Eta and Epsilon} and obtain:
\begin{eqnarray*}
\sum_{j=1}^{\rho} \epsilon_j(jd-a_1) &\geq &
\sum_{j=1}^{\rho} \eta_j(jd-a_1) - (\rho d-a_1) = 
\sum_{j=1}^{\rho} \left( \left \lfloor \frac{\omega_j}{a_1} \right \rfloor - \left \lfloor \frac{\omega_{j-1}}{a_1}  \right\rfloor \right)(jd-a_1) - (\rho d-a_1)  \\
& = &
 \left\lfloor \frac{\omega_{\rho}}{a_1} \right \rfloor (\rho d-a_1)-d\sum_{j=1}^{\rho-1}  \left\lfloor \frac{\omega_{j}}{a_1} \right \rfloor - (\rho d-a_1)
\geq
-d\sum_{j=1}^{\rho-1}  \left\lfloor \frac{\omega_{j}}{a_1} \right \rfloor
\end{eqnarray*}
where we used that $\omega_0=0$ and $ \omega_\rho> \omega_1 > a_1$.
By the assumption on the factors of  $a_1$  we can use Lemma \ref{Lemma with X} for each $j< \rho$, 
yielding
$$
\sum_{j=1}^{\rho-1}  \left \lfloor \frac{\omega_{j}}{a_1}  \right \rfloor
\leq 
 \sum_{j=1}^{\rho-1} \left(j  \left\lfloor \frac{\omega_1}{a_1}\right  \rfloor 
+j - 1\right)
=
{\rho \choose 2 } \left\lfloor \frac{\omega_1}{a_1} \right \rfloor  + {\rho -1\choose 2 } 
$$
and therefore
$$
\sum_{j=1}^{\rho} \epsilon_j(jd-a_1) \geq 
-\frac{d(\rho^2-\rho)}{2}\left \lfloor \frac{\omega_1}{a_1} \right \rfloor -\frac{d(\rho^2 - 3\rho +2)}{2}.
$$
Moreover, since
$
\frac{1}{2}d(\rho^2-\rho)  \le   \frac{1}{2}d(\rho^2-\rho) +\rho d-a_1=  \frac{1}{2}d(\rho^2+\rho) -a_1= ( z+1) d-a_1
$
then
$$
\sum_{j=1}^{\rho} \epsilon_j(jd-a_1) \geq 
-( (z+1) d-a_1) \left\lfloor \frac{\omega_1}{a_1} \right \rfloor   -\frac{d(\rho^2 - 3\rho +2)}{2}.
$$
For the second part of the summation in $\Delta$,
as $z+1\geq \rho +3 $ we can write
$$
\sum_{j=\rho+1}^{a_1-1} \epsilon_j(jd-a_1) \ge\sum_{j=z+1}^{a_1-1} \epsilon_j(jd-a_1) \ge \sum_{j=z+1}^{a_1-1} \epsilon_j((z+1)d-a_1) =((z+1)d-a_1) \sum_{j=z+1}^{a_1-1} \epsilon_j \geq$$
$$
\ge  ((z+1)d-a_1) \left(\sum_{j=z+1}^{a_1-1} \eta_j-1\right)=
((z+1)d-a_1)
\left (  \left \lfloor \frac{\omega_{a_1-1}}{a_1} \right \rfloor- \left \lfloor \frac{\omega_{z}}{a_1} \right \rfloor-1 \right)
$$
where we used again both properties of Lemma \ref{Lemma Properties of Eta and Epsilon} in the last inequality.
Combining the two parts gives
$$
\Delta = \sum_{j=1}^{a_1-1} \epsilon_j(jd-a_1)+(n_Qd-R) \ge 
$$
$$
\left( \left \lfloor \frac{\omega_{a_1-1}}{a_1}\right  \rfloor - \left \lfloor \frac{\omega_{z}}{a_1}  \right\rfloor - \left \lfloor \frac{\omega_{1}}{a_1}  \right\rfloor- 1 \right ) ((z+1)d-a_1)
-\frac{d(\rho^2 - 3\rho +2)}{2}
+ (n_Qd-R)  =: \Pi .
$$
Since $ z \leq y$, 
by Lemma \ref{Lemma With Y} we have the inequality 
$ \lfloor \frac{\omega_{a_1-1}}{a_1}  \rfloor\geq \lfloor \frac{\omega_{z}}{a_1}  \rfloor+ \lfloor \frac{\omega_1}{a_1}  \rfloor$.
If the equality holds, we can use the bound for $n_Q$ in Lemma \ref{Lemma with NQ}
obtaining  
  $$
  \Delta\geq
\Pi
\geq 
-\big((z+1)d-a_1\big)
-\frac{d(\rho^2 - 3\rho +2)}{2}
+ ((y-z+3)d-R) 
=
$$
$$
=
d\left(y-2z-\frac{\rho^2 - 3\rho +2}{2}
+2\right)+(a_1-R)
=
d+(a_1-R)
> 0
$$
by the definitions of $y$ and $z$ and the fact that  $a_1 \geq R$.
Suppose now we have the strict inequality
$ \lfloor \frac{\omega_{a_1-1}}{a_1}  \rfloor > \lfloor \frac{\omega_{z}}{a_1}  \rfloor+ \lfloor \frac{\omega_1}{a_1}  \rfloor$.
In this case the first piece in $\Pi$ is non-negative because $(z+1)d-a_1 \geq \rho d-a_1\geq 0$, thus we can ignore it:
$$
\Pi \ge - \frac{d(\rho^2 - 3\rho +2)}{2}+ (n_Qd-R).
$$
If $n_Q \ge \frac{1}{2}(\rho^2-\rho+4)$,
then $\Delta \geq \Pi \geq \rho d - R +d\geq \rho d - a_1 +d >  0$.
Suppose finally that
 $n_Q \leq \frac{1}{2}(\rho^2-\rho+2)$. 
By Lemma \ref{Lemma with NQ minus 1} we know that
   $n_{Q-1} \ge y+2-n_Q , 
   $
i.e.    $\Sw\cap I_{Q-1}$ contains at least $y+2-n_Q$ elements, 
and in this case from $\rho \geq 3$ it follows
 $$
 y+2-n_Q\geq 
\frac{3\rho^2-\rho-4}{2} +2 -  \frac{\rho^2-\rho+2}{2}
 =
 \rho^2-1
\geq
\frac{\rho^2}{2}+\frac{3\rho}{2}-1 
\geq 
\frac{\rho^2+\rho-2}{2} +\rho
\geq
z+3. 
$$
Let $j_0 = n_{Q-1} \geq y + 2 - n_Q  \geq z + 3 \geq \rho + 1$, we have $\epsilon_{j_0} \geq 1 $ and therefore 
$$
\sum_{j = \rho+1}^{a_1-1} \epsilon_j (jd -a_1) \geq \sum_{j = z+1}^{a_1-1} \epsilon_j (jd -a_1) \geq (j_0 d -a_1) - ((z+1)d -a_1) + \sum_{z+1}^{a_1-1} \epsilon_j((z+1)d-a_1)
$$
it follows that 
$$
\Delta = \sum_{j=1}^{a_1-1} \epsilon_j(jd-a_1)+(n_Qd-R) \ge  \Pi +(y+2-n_Q-z-1)d \ge 
$$
$$
 - \frac{d(\rho^2 - 3\rho +2)}{2}-R+ (y-z+1)d
= \frac{d(\rho^2 + \rho -2 )}{2}-R\geq 
\frac{d(3\rho + 3 -2 )}{2}-R
\geq
(\rho d - R) +\frac{1}{2}
>0
$$
where we used that $\rho \geq 3$ and $\rho d \geq a_1\geq R$.

We have shown that $\Delta>0$ in each case and the theorem is thus proved.
\end{proof}

\section{Two further problems}

Whenever we have a bound, 
it is natural to ask ourselves whether the bound is sharp
and, in case it is, 
to try to characterize the instances in which the bound is achieved.
In his original paper \cite{Wilf}, 
Wilf also asked whether the equality in (\ref{Wilf's Inequality}) is attained if and only if 
$d=a_1$ and $a_i = a_1 +(i-1)$ for $i=2, \ldots, a_1$.
This is  not  the case:
a simple counterexample is $a_1=3, a_2=5$.
However, we believe that equality can only occur in two cases.

\begin{question}\label{Question Equality}
Is it true that the equality $F+1 = d (F+1-g)$ 
holds if and only if either $d = 2$ or
 $d = a_1$ and there exists $K \in \mathbb{N}$ such that $a_i = K a_1 +(i-1)$ for $i=2,\ldots, a_1$?
\end{question}
Note that in the latter case the numerical semigroup has the form 
$$
\mathcal{S}=\{0, a_1, 2a_1, \ldots,(K-1)a_1, K a_1, Ka_1 +1, Ka_1 +2, \ldots\}
$$
and the equality follows from Proposition \ref{Proposition Expression Difference}, while in the former case the equality 
was already known to Sylvester (cf. \cite[2.12]{RosalesGarciaSanchez}),
so the interesting part is the ``only if''.
We observe that, in order to show that these are the only two cases,
it  suffices to prove that either $d=2$ or $d=a_1$. 
In fact, if  $d=a_1$ and $F+1 = d(F+1-g)$ then     Proposition \ref{Proposition Expression Difference} implies the equation
$$
\Delta=\sum_{j=1}^{a_1-1} \epsilon_j(j-1)a_1+(n_Q a_1-R)=0 
$$
and since $n_Q\geq 1, a_1 \geq R$ we conclude that 
 $\epsilon_j = 0$ for every $ j \geq 2$, $n_Q=1, R=a_1$;
  it follows that $a_i = (Q+1) a_1 +(i-1)$ for $i=2,\ldots, a_1$.
  
We have verified via \texttt{GAP} (cf. \cite{NS}) that Question \ref{Question Equality} 
has an affirmative answer when $g \leq 35$.
Further evidence in its favor is given by the proof of  Theorem \ref{Main Theorem},
as in all the cases investigated therein 
(where $\rho \geq 3 $, hence $d < a_1$)
the strict inequality was actually seen to hold: 
in other words, 
we give a positive answer under the assumptions of our theorem.
\\

Finally, 
we remark that the conjecture has an interpretation in Commutative Algebra in terms of length inequalities.
Let $ \mathcal{R} = \Bbbk [[ t^{a_1}, \ldots , t^{a_d} ]] $
be the local ring of a monomial curve, 
where $\Bbbk$ is a field. 
Let $\overline{\mathcal{R}}=\Bbbk[[t]]$ be the integral closure of $\mathcal{R}$ in its field of fractions $\mathcal{Q}=\Bbbk((t))$ and
 $\mathfrak{C}=(\mathcal{R} :_\mathcal{Q} \overline{\mathcal{R}})$ be the conductor of $\mathcal{R}$ in $\overline{\mathcal{R}}$, 
 that is the largest common ideal of $\mathcal{R}$ and $\overline{\mathcal{R}}$.
Denoting by $\ell(\cdot)$ the length of an $\mathcal{R}$-module,
 the values of $\ell( \overline{\mathcal{R}}/\mathcal{R})$ and  $\ell(\mathcal{R}/\mathfrak{C})$ are both measures of the singularity of $\mathcal{R}$ and they are related by $\ell(\mathcal{R}/\mathfrak{C}) \leq \ell( \overline{\mathcal{R}}/\mathcal{R}) $, with equality holding if and only if $\mathcal{R}$ is Gorenstein.
 Under the notation of this paper, 
the embedding dimension of $\mathcal{R}$ is $\edim(\mathcal{R})=d$, 
whereas  $\ell( \overline{\mathcal{R}}/\mathfrak{C})=F+1$ and  $\ell(\mathcal{R}/\mathfrak{C})=F+1-g$ 
  (cf. \cite[II.1]{BarucciDobbsFontana}).
Therefore (\ref{Wilf's Inequality}) is equivalent to 
\begin{equation}\label{Wilf's Length Inequality}
\ell( \overline{\mathcal{R}}/\mathfrak{C}) \leq \edim(\mathcal{R}) \ell(\mathcal{R}/\mathfrak{C}).
\end{equation}
A similar well-known inequality holds in a  more general context: 
 if $\mathcal{R}$ is a 
one-dimensional Cohen-Macaulay local ring with  
 type $t(\mathcal{R})$ and such that 
the integral closure 
$\overline{\mathcal{R}}$
of $\mathcal{R}$ in its total ring of fractions is a finite $\mathcal{R}$-module, 
then
\begin{equation}\label{TypeInequality}
\ell( \overline{\mathcal{R}}/\mathfrak{C})  \leq (t(\mathcal{R})+1) \ell(\mathcal{R}/\mathfrak{C})
\end{equation}
 see e.g. \cite{BrownHerzog},  \cite{Delfino}, \cite{Matsuoka}.
It follows that  (\ref{Wilf's Length Inequality}) is satisfied if $t(\mathcal{R})<\edim(\mathcal{R})$, 
but the two invariants are unrelated in general:
J. Backelin exhibited a family of rings with $\edim(\mathcal{R})=4$ and arbitrarily large type (cf. \cite{FGH}).
It would be natural to explore 
 the form (\ref{Wilf's Length Inequality}) of Wilf's inequality for more general classes of rings,
therefore we conclude the paper with the following general question.

\begin{question}
Let $\mathcal{R}$ be a 
one-dimensional Cohen-Macaulay local ring such that the integral closure of $\mathcal{R}$ in its total ring of fractions is  a finite $\mathcal{R}$-module.
Under what assumptions do we have
$
\ell( \overline{\mathcal{R}}/\mathfrak{C}) \leq \edim(\mathcal{R}) \ell(\mathcal{R}/\mathfrak{C})$?
\end{question}

\end{document}